\documentclass[a4paper, 12pt, reqno]{amsart}
\usepackage[left=1.5cm,right=1.5cm, top=1.5cm, bottom=2cm]{geometry}
\usepackage[latin1]{inputenc}
\usepackage[T1]{fontenc}
\usepackage{indentfirst}
\usepackage{latexsym}
\usepackage[none]{hyphenat}
\usepackage{amsmath}
\usepackage{amsthm}
\usepackage{amsfonts}
\usepackage{mathrsfs}
\usepackage{amssymb}
\usepackage{wasysym}
\usepackage{enumerate}
\usepackage{color}
\usepackage{graphicx}

\usepackage[centertags]{amsmath}
\usepackage[english]{babel}
\usepackage[active]{srcltx}

\usepackage[colorlinks]{hyperref}

\newcommand{\dist}{\operatorname{dist}}

\newcommand{\rank}{\operatorname{rank}}

\newcommand{\conv}{\operatorname{conv}}

\newtheorem{Teo}{\textbf{Theorem}}[section]
\newtheorem{Pro}[Teo]{\textbf{{Proposition}}}

\newtheorem{Def}[Teo]{\textbf{{Definition}}}
\newtheorem{Lem}[Teo]{\textbf{{Lemma}}}
\newtheorem{Step}{\textbf{{Step}}}

\newtheorem{Rem}[Teo]{\textbf{{Remark}}}

\numberwithin{equation}{section}

\newcommand{\tresp}[1]{{\left\vert\kern-0.25ex\left\vert\kern-0.25ex\left\vert #1
    \right\vert\kern-0.25ex\right\vert\kern-0.25ex\right\vert}}

\title[Normal and starlikes tilings]{Normal and starlike tilings in separable Banach spaces}

\author{Robert Deville and Mar Jimenez-Sevilla}

\date{July, 2020}

\address{Robert Deville, Institut de Math\'ematiques de Bordeaux UMR 5251, Universit\'e de Bordeaux 1, 33405 Talence, France}

\email{Robert.Deville@math.u-bordeaux.fr}

\address{Mar Jimenez-Sevilla, Instituto de Matem\'atica Interdisciplinar y Departamento de An\'alisis Matem\'atico y Matem\'atica Aplicada, Facultad de Ciencias Matem\'aticas, Universidad Complutense de Madrid, Madrid 28040, Spain}

\email{marjim@ucm.es}

\thanks{%R. Deville was supported by the Grants {\color{red} (COMPLETE)}.
M. Jimenez-Sevilla  has been partially supported  by Ministerio de Ciencia, Innovaci\'on y Universidades project 
PGC2018-097286-B-I00 (Spain) and  grant PRX19-00094-Estancias de profesores e investigadores s\'enior en centros extranjeros, incluido el Programa Salvador de Madariaga (Spain).}



\begin{document}

\begin{abstract} In this note, we provide a starlike and normal  tiling in any separable Banach space. That means, there are positive constants $r$ and $R$ (not depending on the separable Banach space) such that every tile of this tiling is starlike,  contains  a ball of radius $r$ and   is contained in a ball of radius $R$. 
\end{abstract}

\maketitle

%\begin{figure}
 % \includegraphics[width=8cm]{ScanPictures3.pdf}
 % \caption{Figure 3}
 % \label{fig:Figure 3}
%\end{figure}

%Figure \ref{fig:Figure 3} shows........

\section{Introduction}

 A family  $\{T_i\}_{i\in I}$ of subsets of a Banach space $X$ is a {\em tiling of } $X$ if $\bigcup_{i\in I}T_i=X$,
each tile $T_i$ is a closed  subset of $X$ with nonempty interior and  $C_i\cap C_{j}$ has  empty interior for all $i\not= j$.  
The tiling is said to be convex (starlike)  whether the tiles $T_i$ are convex (starlike, respectively).

In the following, all Banach spaces are considered over the reals. We use the standard notation in a Banach space $X$ with norm 
$||\cdot||$. The closed ball of center $x\in X$ and radius $r>0$   in $X$ is denoted by $B(x,r)$ (also by $B_{||\cdot||}(x,r)$ or 
 $B_{X}(x,r)$ if it is necessary to precise the norm or the space, respectively, we are referring to).

The study of tilings in  infinite dimensional Banach space started with V. Klee in \cite{Klee1} and \cite{Klee2}. V. Klee constructed  in every separable Banach space $X$  by means of biorthogonal systems a convex tiling by ``parallelotopes'' that are {\em finitely bounded} (bounded intersections with each finite dimensional subspace) but norm-unbounded.
He also studied the existence of tilings based on {\em discrete proximinal} sets in some Banach spaces. 
Suppose that a set $A\subset X$ is a discrete proximinal set and define the associated ``Voronoi cells'' 
$$T_a=\{x\in X:\,||x-a||\le ||x-a'|| \text{ for all } a'\in A\}.$$ 
V. Klee proves that the family of sets $\{T_a\}_{a\in A} $ is a starlike tiling whenever the norm is strictly convex.
Moreover, if $X$ is a Hilbert space and the norm $||\cdot||$ is a given by a scalar product in $X$, then the tiling is convex.

\begin{Def} A tiling $\{T_i\}_{i\in I}$  of a Banach space $(X,||\cdot||)$ is said to be  $K$-normal if there are constants $r>0$,  $K>0$ and 
 points $x_i\in T_i$ satisfying $B(x_i,r)\subset T_i \subset B(x_i,rK)$ for all $i$.
\end{Def}
For any  $K$-normal tiling $\{T_i\}_{i\in I}$ and for every $s>0$ the homothetic family $\{sT_i\}_{i\in I}$ satisfies $B(sx_i,sr)\subset sT_i \subset B(sx_i,srK)$ for all $i$ and thus $\{sT_i\}_{i\in I}$ is a $K$-normal tiling of $(X,||\cdot||)$ as well. 
Notice that the constant $K$ of normality depends on the norm. By renorming with an equivalent norm $|||\cdot|||$ such that $|||x|||\le||x||\le M|||x|||$ for some $M>0$ and all $x\in X$,
then $B_{|||\cdot|||}(x_i,\frac{r}{M})\subset T_i \subset B_{|||\cdot|||}(x_i,rK)$ for all $i$ and the constant of normality of the tiling becomes $KM$. A tiling is said to be  normal if it is $K$-normal for some $K>0$.

For example, for every $n\in \mathbb N$, in $\ell_2^n$ with the euclidean norm, the ``Voronoi cells'' associated with a  maximal $2$-separated set is a $2$-normal convex tiling. R. Deville and M. Garc\'\i a-Bravo proved that in any finite dimensional space there is a $2$-normal starlike tiling \cite{DevilleGarciaBravo}.
Also, the space $(c_0(\Gamma),||\cdot||_\infty)$ (for any non empty set $\Gamma$) has a tiling by closed balls of radius $1$ with centers at the points having coordinates in $2\mathbb Z$ with only a finite number of non-null coordinates. The Banach space
$(\ell_\infty(\Gamma),||\cdot||_\infty)$ (for any non empty set $\Gamma$) has a tiling by closed balls of radius $1$ with centers at points having coordinates in $2\mathbb Z$ as well.
V. Klee proved the existence in the space $(\ell^{p}(\Gamma),||\cdot||_p)$ with $1\le p<\infty$ and for any infinite cardinal $\Gamma$ satisfying $\Gamma^{\aleph_0}=\Gamma$ the existence of a $2^{1/p}$-separated proximinal set $A$ such that the ``Voronoi cells'' associated with the set $A$ form a $2^{(p-1)/p}$-normal starlike tiling of the space $\ell^{p}(\Gamma)$. Moreover, when $p=2$ the tiling is convex and when $p=1$ the tiles  are pairwise disjoint translates of the closed unit ball $B_{||\cdot||_1}(0,1)$.

Later on, V. Fonf, A. Pezzota and C. Zanco \cite{FPZ} constructed for any infinite set $\Gamma$
 a convex tiling $\{T_i\}_{i \in \Gamma}$ in $\ell_\infty(\Gamma)$ with bounded tiles which is universal in the sense that
for any normed space $X$ with a norming set  $S\subset X^*$ of norm one functionals of cardinal $|\Gamma|$,  there is an isomorphic embedding $I:X\rightarrow \ell_\infty(\Gamma)$ such that $\{I^{-1}(T_i)\}_{i\in \Gamma}$ is a convex tiling of $X$ with bounded tiles 
and inner radii uniformly bounded from below by a positive constant, i.e. there is a constant $r>0$ and there are points $x_i \in T_i$ satisfying $B(x_i,r)\subset T_i$ for all $i$. Nevertheless the diameter of the tiles are not uniformly bounded from above. So the tiling $\{I^{-1}(T_i)\}_{i\in \Gamma}$ of $X$ is not normal.

For more properties on tilings such as point-finite, locally finite, protective,  smooth or strictly convex, see \cite{C}, \cite{DBV}, \cite{Fonf}, \cite{FPZ2}, \cite{FZ}, \cite{FZ2}  and references thererin.

It was proved by V. Fonf and J. Lindenstrauss  \cite{FL} that in the separable Hilbert space $\ell_2$ any
maximal  $1$-separated  set  is not proximinal. So the construction of convex and normal tilings with ``Voronoi cells'' given by V. Klee is not possible in $\ell_2$. Later on, D. Preiss \cite{Preiss} proved the existence of  convex and normal tilings in $\ell_2$. Recently, R. Deville and M. Garc\'\i a-Bravo proved the existence of normal  and starlike tilings in any separable Banach space with a Schauder basis \cite{DevilleGarciaBravo}. In this note, we reexamine the construction of convex  normal tilings given by D. Preiss in $\ell_2$  and the construction of unbounded convex tilings 
with inner radii uniformly bounded from below by a positive constant in general separable Banach spaces given also by Preiss \cite{Preiss}. As a result we get 
normal  and starlike tilings (with the same constant of normality)  for any separable Banach space $(X,||\cdot||)$.
 Thus, the construction provides a universal constant of normality, i.e. it can be taken  not to depend on the 
(separable) space nor on the norm. Also it is worth to mention that these normal and starlike tilings can be defined in such a way that some  of the tiles are convex.

 Recall that the Banach-Mazur distance from any $n$-dimensional space $X$ to $\ell_2^n$ satisfies $\dist(X,\ell_2^n)\le\sqrt{n}$. Since $\ell_2^n$ has a  $2$-normal convex tiling, then $X$ has a $2\sqrt{n}$-normal convex  tiling. The question whether every separable Banach space $X$ (finite or infinite dimensional) has a $K$-normal convex tiling with constant of normality $K$ independent of the space and  the norm remains open. We do not know whether a positive answer to the finite dimensional case provides a positive answer to the infinite dimensional case.

\section{Normal and starlike tilings in separable Banach spaces}

Let start with a lemma providing an auxiliary tiling in $\mathbb R^2$.

\begin{Lem}\label{R2}
 Let us consider the subsets of $\mathbb R^2$,
%\begin{equation*} 
\begin{align*}
D&=[-1,1]\times [-1,1],\quad U_0=\{(x,y)\in \mathbb R^2:\, |x|+|y|\le 2\}, \\
U_1&=\{(x,y)\in \mathbb R^2:\,0\le x\le 2;\, y\ge 0; \, x+y\ge 2\}, \quad
 U_2=\{(x,-y)\in \mathbb R^2:\,(x,y)\in U_1\},\\    U_3&=-U_1, \quad   U_4=-U_2.
\end{align*}
%\end{equation*}
Then,  the family $\{U_i:\,i=0,1,2,3,4\}$ is a tiling of $\{(x,y)\in \mathbb R^2:\,|x|\le 2\}$ and for any $1<a<2$ and $0<b<1$
with $a+b>2$ there are constants $0<r,\delta<1$ satisfying
\begin{enumerate}
\item[(L.a)] $D\subset U_0$,
\item[(L.b)] $(a,b)+rD\subset U_1\cap \{(x,y)\in \mathbb R^2:\,|y|\le 1\}$,
\item[(L.c)]  $(a,t)+rD\subset U_0$ whenever $|t|\le \delta b$.
\end{enumerate}
\end{Lem}
\begin{proof} The computations on the conditions given for the squares $(a,b)+rD$ and $(a,t)+rD$ whenever $|t|\le \delta b$ 
(see Figure \ref{picturetilingR2-modelo3})  yield   the following conditions on $r$ and $\delta$:
 \linebreak
$0<r\le \min\{1-b,\,\frac{a+b}{2}-1,\,2-a\}$,\  $r<1-\frac{a}{2}$ \ and \ $0<\delta \le  \frac{2-2r-a}{b}$.   
\end{proof}
\begin{figure}
 \includegraphics[width=5cm]{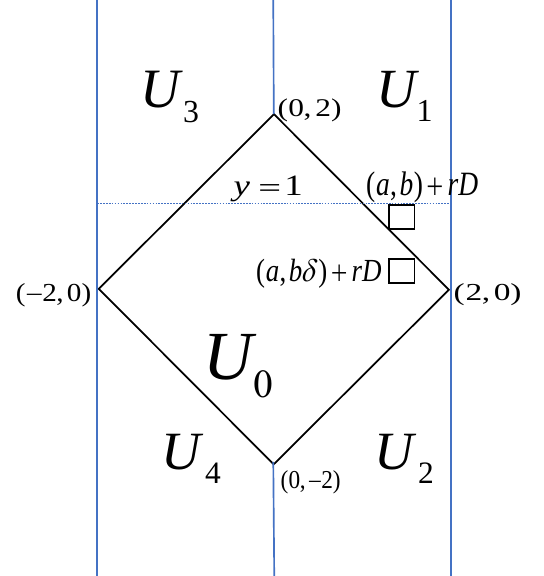}
 \caption{Tiling in $\mathbb R^2$}
  \label{picturetilingR2-modelo3}
\end{figure}

\begin{Teo} \label{separablecase} Let $(X,||\cdot||)$ be a separable Banach space. Then, there is a $K$-normal and starlike tiling of $X$  
for any constant $K>2+\frac{1}{r}(a+2b+\frac{8}{\delta})$ (where the constants $a,b$, $r$ and $\delta$ are given in Lemma \ref{R2}). 
\end{Teo}

 It is worth to mention that some  tiles of this tiling are convex (see Remark \ref{k=0}, Propositions \ref{separableconvex} and \ref{Ndim} and Remark \ref{finitedimensionalconvex} regarding the convexity of some  tiles).

\bigskip

\noindent {\em Proof of Theorem \ref{separablecase}.} Since it is already known that finite dimensional spaces have $2$-normal starlike tilings \cite{DevilleGarciaBravo}, we may assume that $X$ is infinite dimensional. The proof consists on a refinement of the convex tiling obtained by D. Preiss given in \cite[Theorem 9]{Preiss} by constructing a starlike tiling in each of those tiles (Step \ref{refinement}). In order to give a complete proof of the theorem we will provide as well the proof of several lemmas given in \cite{Preiss} adapted to our particular case and  with some modifications in the constants. 

\medskip

Let us start the proof by considering $\{e_i,e_i^*\}_{i=1}^\infty$ a fundamental biorthogonal system in $X$. For any $\varepsilon>0$ we may assume  up to a $(1+\varepsilon)$-renorming that 
$||e_i^*||=||e_i||=1$ for all $i$. Indeed, from the result given in \cite[Corollary 1.26]{HMVZ}, we get a fundamental biorthogonal system $\{a_i,a_i^*\}_{i=1}^\infty$ such that $||a_i||||a^*_i||\le 1+\varepsilon$ for all $i$.
 By defining $e^*_i=\frac{a^*_i}{||a^*_i||}$ and   $e_i={||a^*_i||}{a_i}$ for all $i$ we get 
$||e^*_i||=1=e^*_i(e_i)$  and $1\le ||e_i||\le 1+\varepsilon$ for all $i$. Considering the $(1+\varepsilon)$-equivalent norm 
$|\cdot|$
whose closed unit ball is 
$B_{|\cdot|}=\overline{\conv}(B_{||\cdot||}\cup\{\pm e_i\})$ we get the normalized fundamental biorthogonal system 
$\{e_i^*,e_i\}$ for the norm $|\cdot|$. In the following we use same notation $||\cdot||$ for the new norm.

\medskip

Let us define for 
$k\in \mathbb N\cup \{0\}$ the finite dimensional subspaces, $V_k=[\{e_1,\cdots,e_k\}]$ the linear span of $e_1,\dots,e_k$
with $V_0=\{0\}$. Also, we define the quotients $Z_k=X/V_k$ with the corresponding quotient norm $||\cdot||_{Z_k}$ for every $k\in \mathbb N\cup\{0\}$. If there is no ambiguity on the norm we are referring to  we shall write only $||\cdot||$ for the quotient norm. Consider la natural projections
\begin{equation*} 
Q_{k}:X\rightarrow Z_{k} , \quad Q_{k}(\widehat{x})=\widehat{x},\
\end{equation*}

\begin{equation*}
Q_{k,k+1}:Z_k\rightarrow Z_{k+1} , \quad Q_{k+1}(\widehat{x})=\widehat{x},
\end{equation*}
For simplicity, we refer to the corresponding equivalence classes in $Z_k$ for all $k$ in the same way.

\medskip

\begin{Step} \label{semibetasystem} Consider $Z_{k}=X/V_k$ with the canonical quotient norm and a constant $0<\delta<1$. Then, there is a family 
$\{v^*_{j,k},v_{j,k}\}_{j=0}^\infty \subset Z_{k}^*\times Z_{k}$ such that 

\begin{enumerate}

\item[(1.a)] $||v^*_{j,k}||=||v_{j,k}||=1=v_{j,k}^*(v_{j,k})=1$ for all $j,k$.

\item[(1.b)]  $|v_{j,k}^*(v_{j',k})|\le \delta$ if $j<j'$.

\item[(1.c)]  $ \sup_j|v_{j,k}^*(v)|\ge \delta ||v||$ for all $ v\in Z_{k}$.

\end{enumerate} 
\end{Step}
\begin{proof} % \let\qed\relax
Several proofs are given in \cite[Lemma 7]{Preiss} and \cite[Lemma 2.5]{DevilleGarciaBravo}.
Let us give here the following proof. Consider a dense set $\{x_n\}_{n=0}^\infty$ in the unit sphere of $Z_k$ and $\{x_n^*\}_{n=0}^\infty$ the associated functionals with norm one such that
$x_n^*(x_n)=1$. Define $v_0=x_0$,\,  $v_0^*=x_0^*$ and $n_0=0$. Take  the first natural (if it exists) such that  $n_1>n_0$
 and $|v_0^*(x_{n_1})|\le \delta$ and define $v_1=x_{n_1}$ and $v_1^*=x_{n_1}^*$. In general, take $n_{j+1}$ the first natural (if it exists) greater than $n_j$ such that 
$|v_i^*(x_{n_{j+1}})|\le \delta$ for $i\le j$ and define $v_{j+1}=x_{n_{j+1}}$ and $v_{j+1}^*=x_{n_{j+1}}^*$. 
If the space $Z_k$ is finite dimensional the above process will stop at some step $j_0$ and the family $\{v^*_j,v_j\}_{j=1}^{j_0}$ would satisfy properties (1.a)-(1.c). In our case $Z_k$ is infinite dimensional and the above process continues for every $j$ (notice that
$\cap_{i\le j}\ker v^*_i$ is a non-trivial subspace) and it produces an infinite sequence 
$\{v_j^*,v_j\}_{j=0}^\infty$ satisfying properties (1.a)-(1.c). Indeed, it is clear that properties (1.a) and (1.b) hold. To prove property (1.c) it is enough to see that (1.c) holds for the points of the sequence $\{x_m\}_{m=0}^\infty$ because this sequence is dense in the sphere.
For any  $m\ge 0$ there is $i\ge 0$ such that $n_i\le m<n_{i+1}$. If $n_i=m$, then clearly $\sup_j|v_{j}^*(x_m)|=|v_i^*(x_m)|=1> \delta ||x_m||$.
If $n_i<m<n_{i+1}$,  there is $s\le i$ such that $|v_s^*(x_m)|>\delta=\delta ||x_m||$ and thus $\sup_j |v_j^*(x_m)|>\delta ||x_m||$. This yields property (1.c). Finally, we relabel  $\{v_j^*,v_j\}_{j=0}^\infty$ as $\{v^*_{j,k},v_{j,k}\}_{j=0}^\infty$ and finish the proof of Step \ref{semibetasystem}.
% Consider the family $\mathcal F$ of all countable (finite or infinite) families  $\{(w_i^*,w_i)\}_{i=1}^{\tau}\subset X^*\times X$
%where $\tau \in \mathbb N\cup \{\infty\}$ such that $||w_i^*||=||w_i||=1=w_i^*(w_i)$ and $|w_i^*(w_j)|<\delta$ for $i<j$ ordered by the relation $\{(w_i^*,w_i)\}_{i=1}^{\tau}\le
 %\{(\widetilde{w}_i^*,\widetilde{w}_i)\}_{i=1}^{\lambda}$, where $\tau, \,\lambda \in \mathbb N\cup \{\infty}$ if $\tau \le \lambda$ and $w_i^*=\widetilde{w}_i^*$, \, $w_i=\widetilde{w}_i$ for all $i\le \tau$ 
%(or $i\in \mathbb N$ if $\tau=\lambda=\infty$). Every chain in $\mathcal F$ has a upper bound  in $\mathcal F$ and by Zorn's Lemma $\mathcal F$ has a maximal element. This maximal element satisfies the properties (1.a)-(1.c).
\phantom\qedhere
 \end{proof} 

\medskip

\begin{Step}\cite[Lemma 8]{Preiss} \label{1scaling} There is a convex tiling $\{H^k_j\}_{j=0}^\infty$ of $Z_k=X/V_k$ and points $h^k_j\in H^k_j$ with $h^k_0=0$ such that
\begin{enumerate}
\item[(2.a)] $B_{Z_k}(0,1)\subset H_0^k\subset B_{Z_k}(0,\frac{4}{\delta})$,
\item[(2.b)] $||Q_{k,k+1}h^k_j||\le 1-r$,
\item[(2.c)] $B_{Z_k}(h^k_j,r)\subset H^k_j$,
\item[(2.d)] \label{c.4} $||h-h^k_j||\le a+2b+2||Q_{k,k+1}h||$ for every $h\in H^k_j$,
\end{enumerate}
where the positive constanst $\delta$, $r$, $a$ and $b$ are those considered in Lemma \ref{R2} and Step \ref{semibetasystem}.
\end{Step}
\begin{proof} %\let\qed\relax
  It can be easily checked that  
$Q_{k,k+1}(B_{Z_k}(0,1))= B_{Z_{k+1}}(0,1)$, where $B_{Z_k}(0,1)$ denotes the closed unit ball of $Z_k$. Indeed, on the one hand 
$||Q_{k,k+1}(\widehat{x})||=\dist(x,V_{k+1})\le \dist(x,V_k)= ||\widehat{x}||$
 for all $x\in X$ and thus $Q_{k,k+1}(B_{Z_k}(0,1))\subset B_{Z_{k+1}}(0,1)$. 
%On the other hand, there is  $y\in S_X$ ($S_X$ denotes the unit sphere of $X$) such that
%$\dist(y,V_{k+1})=1$ and thus $\dist(y,V_{k})=1$. This yields $||Q_{k,k+1}(\widehat{y})||=||\widehat{y}||$.
On the other hand, clearly $Q_{k,k+1}\circ Q_k=Q_{k+1}$ for all $k\ge 0$ and $Q_k(B(0,1))=B_{Z_{k}}(0,1)$ for all $k\ge 0$.
Therefore, $Q_{k,k+1}(B_{Z_k}(0,1))= B_{Z_{k+1}}(0,1)$ for all $k\ge 0$. 
Also clearly $\ker Q_{k,k+1}=\{\widehat{x}\in Z_k:\, x\in V_{k+1}\setminus V_k\}\cup \{0\}$ so the dimension of $\ker Q_{k,k+1}$ is $1$.

\medskip

Let us define for $j\ge 0$ and $k\ge 0$, the functions
\begin{equation*}
\pi_{j,k}:Z_k\rightarrow \mathbb R^2, \quad \pi_{j,k}(z)=(e^*_{k+1}(z), v^*_{j,k+1}\circ Q_{k,k+1}(z)), \quad \text{ for }
z\in Z_k.
\end{equation*}
Here, we identify isometrically $(Z_{k})^*=V_k^\perp=\{f\in X^*:\,f(x)=0 \text{ for all } x\in V_k\}$. In particular, notice that
$\widehat{e}_{k+1}$ has norm $1$ as a point of $Z_k$  and $e_{k+1}^*$ has norm $1$ as a funcional on $Z_k$.

For every $z\in Z_k$ we select a point $x\in X$ such that $z=\widehat{x}$ and a linear combination
 $\sum_{i=1}^{k+1}\lambda_ie_i\in V_{k+1}$ such that 
\begin{equation*}
\begin{split}
||Q_{k,k+1}(z)||_{Z_{k+1}}&=\dist(x,V_{k+1})=||x-\sum_{i=1}^{k+1}\lambda_ie_i||=||(x-\lambda_{k+1}e_{k+1})-\sum_{i=1}^{k}\lambda_ie_i||\ge \\
&\ge\dist(x-\lambda_{k+1}e_{k+1}, V_k)=||z-\lambda_{k+1}\widehat{e}_{k+1}||_{Z_k}.
\end{split}
\end{equation*}
Then,
\begin{equation}\label{inequalityforc4}
\begin{split}
||z||_{Z_k}&\le ||e_{k+1}^*(z)\widehat{e}_{k+1}||_{Z_k}+||\lambda_{k+1}\widehat{e}_{k+1}-e_{k+1}^*(z)\widehat{e}_{k+1}||_{Z_k}
+||z-\lambda_{k+1}\widehat{e}_{k+1}||_{Z_k}\le \\
 &\le |e_{k+1}^*(z)|+|e_{k+1}^*(z-\lambda_{k+1}\widehat{e}_{k+1})|+ ||Q_{k,k+1}(z)||_{Z_{k+1}}\le \\
&\le |e_{k+1}^*(z)|+||z-\lambda_{k+1}\widehat{e}_{k+1}||_{Z_k}+ ||Q_{k,k+1}(z)||_{Z_{k+1}} \le \\
&\le |e_{k+1}^*(z)|+2||Q_{k,k+1}(z)||_{Z_{k+1}}. 
\end{split}
\end{equation}

Consider the following sets
%
%
%
%
%\begin{equation*}
\begin{align*}
T_k&=\{z\in Z_k:\,|e^*_{k+1}(z)|\le 2\},\\
H_0^k&= \bigcap_{j\ge 0}\pi_{j,k}^{-1}(U_0),\\
H_{j}^{k,p}&= \pi_{j,k}^{-1}(U_p) \cap\bigcap_{j'<j}\pi_{j',k}^{-1}(U_0),\quad \text{ for } 1\le p \le 4 \, \text{ and } \, j\ge 0.
\end{align*}
%\end{equation*}

\noindent  where the sets $U_0,U_1,\cdots,U_4$ are the subsets of $\mathbb R^2$ defined in Lemma \ref{R2}. Then, let us check
that the family 
$\{H_0^k,H_{j}^{k,p}:\,j\ge 0,\,p=1,2,3,4\}$ is a tiling of $T_k$ satisfying conditions (2.a)-(2.d) in the statement of  Step \ref{1scaling}:

\begin{enumerate}

\item[{\em Tiling.}] Let us check that  the above family is a convex tiling. Clearly the sets of the family $\{H_0^k,H_{j}^{k,p}:\,j\ge 0,\,p=1,2,3,4\}$ are closed and convex because they are the inverse image by continuous linear functions of closed and convex sets. Also this family covers  $T_k$: Indeed,  Suppose that $z\in T_k$ and that there is 
$j\ge 0$ with  $\pi_{j,k}(z)\notin U_0$ and take the minimum with this property.
Then, $z\in H_{j}^{k,p}$ for some $p\in \{1,2,3,4\}$. Otherwise $\pi_{j,k}(z)\in U_0$ for all $j\ge 0$ and then $z\in H^k_0$.

Also, the sets of the family  $\{H_0^k,H_{j}^{k,p}:\,j\ge 0,\,p=1,2,3,4\}$ have pairwise disjoint interiors. Suppose there is a ball $B_{Z_k}(z,\rho)$ (with $\rho>0$) contained in $H_0^k\cap H_{j}^{k,p}$ for some $j\ge 0 $ and  $p\in\{1,2,3,4\}$. 
Then $\pi_{j,k}(B_{Z_k}(z,\rho))\subset U_0\cap U_p$. Since the continuous linear function $\pi_{j,k}$ is surjective, it is open and thus $\pi_{j,k}(B_{Z_k}(z,\rho))$ is a nonempty open set of $\mathbb R^2$, which is impossible (because the set $U_0\cap U_p$ has empty interior in $\mathbb R^2$).
The proof for any pair of sets  $H_{j}^{k,p}$ and $H_{j'}^{k,p'}$ with $(j,p)\not=(j',p')$ is similar.

\medskip

\item[{\em (2.a)}] 
If $z\in B_{Z_{k}}(0,1)$, then for every $j\ge 0$,
\begin{equation*}
\begin{split}
\max \{|e^*_{k+1}(z)|, &\, |v_{j,k+1}^*\circ Q_{k,k+1}(z)|\}\le 
\max \{||e^*_{k+1}||\,||z||,\,||v_{j,k+1}^*||\,||Q_{k,k+1}||\,||z||\}=||z||\le 1.
\end{split}
\end{equation*}
and thus $\pi_{j,k}(B_{Z_{k}}(0,1))\subset D=[-1,1]\times [-1,1]\subset U_0$ for all $j\ge 0$. This means that 
$$B_{Z_k}(0,1)\subset H_0^k.$$

\medskip

\noindent  For the second inclusion, if $z\in H_0^k$ and $j\ge 0$,
\begin{equation*}
|e^*_{k+1} (z)|+ |v_{j,k+1}^*\circ Q_{k,k+1}(z)|\le 2,
\end{equation*}
and thus 
\begin{equation}\label{squareinequalities}
 \sup_j|v_{j,k+1}^*\circ Q_{k,k+1}(z)|\le 2-|e^*_{k+1} (z)|.
\end{equation}
Also recall that 
\begin{equation*}
 \delta||Q_{k,k+1}(z)||_{Z_{k+1}}\le \sup_j|v_{j,k+1}^*\circ Q_{k,k+1}(z)|.
\end{equation*}
Thus from inequalities \eqref{inequalityforc4} and \eqref{squareinequalities} we get
\begin{equation*}
\begin{split}
||z||_{Z_k}&\le |e_{k+1}^*(z)|+2||Q_{k,k+1}(z)||_{Z_{k+1}}\le 
 |e_{k+1}^*(z)|+ \frac{2}{\delta} \sup_j|v_{j,k+1}^*\circ Q_{k,k+1}(z)| \le \\
&\le |e_{k+1}^*(z)|+ \frac{2}{\delta} (2-|e^*_{k+1} (z)|)=(1-\frac{2}{\delta})|e^*_{k+1} (z)|+
\frac{4}{\delta}\le \frac{4}{\delta}.
\end{split}
\end{equation*}
The last inequality holds because $0<\delta<1$ and then $1-\frac{2}{\delta}<0$.
This yields
\begin{equation*}
H^k_0\subset B_{Z_{k}}(0,\frac{4}{\delta}).
\end{equation*}

\bigskip

\item[{\em (2.b)}] Let us select a point $u_{j,k+1}\in Z_k$ such that 
$Q_{k,k+1}(u_{j,k+1})=v_{j,k+1}$ and $e_{k+1}^*(u_{j,k+1})=0$. This can be done by taking $x\in X$ with
$Q_{k+1}(x)=v_{j,k+1}$. Then $v_{j,k+1}=Q_{k+1}(x)=Q_{k,k+1}(Q_{k}(x))=Q_{k,k+1}(Q_{k}(x+V_{k+1}))=
Q_{k,k+1}(Q_{k}(x)+\lambda \widehat{e}_{k+1})$ for any $\lambda\in \mathbb R$.
Clearly there is $\lambda_0$ such that $e^*_{k+1}(Q_{k}(x)+\lambda_0 \widehat{e}_{k+1})=0$.
Then, we can take $u_{j,k+1}=Q_{k}(x)+\lambda_0 \widehat{e}_{k+1}$.
If we define $h_{j}^{k,1}=a\widehat{e}_{k+1}+bu_{j,k+1}\in Z_{k}$, we have 
\begin{equation*}
||Q_{k,k+1}(h_{j}^{k,1})||=b||v_{j,k+1}||=b\le 1-r.
\end{equation*}

\bigskip

\item[{\em (2.c)}] Notice that $\pi_{j,k}(h_{j}^{k,1})=(a,b)$ for all $j\ge 0$ and thus
\begin{equation*}
\pi_{j,k}(B_{Z_k}(h_{j}^{k,1},r))\subset (a,b)+r([-1,1]\times [-1,1])\subset U_1.
\end{equation*}
Moreover, if $j'<j$, then 
\begin{equation*}
\pi_{j' ,k}(B_{Z_k}(h_{j}^{k,1},r))\subset  (a,\,bv_{j',k+1}^*(v_{j,k+1}))+r([-1,1]\times [-1,1])
\subset U_0
\end{equation*}
because the second coordinate
satisfies $|bv_{j',k+1}^*(v_{j,k+1})|\le b\delta$ (we apply condition (L.c)).
Thus
\begin{equation*}
B_{Z_{k }}(h_{j}^{k,1},r)\subset H_{j}^{k,1}, \quad \text{ for all } j\ge 0.
\end{equation*}

\bigskip

\item[{\em (2.d)}] Set $h^k_0=0$. If $z\in H^k_0$ then by  \eqref{inequalityforc4},

\begin{equation*}
||z||\le |e_{k+1}^*(z)|+2||Q_{k,k+1}(z)||\le 2+2||Q_{k,k+1}(z)||.
\end{equation*}

\noindent Now, if $z\in H_{j}^{k,1}$ then by  \eqref{inequalityforc4},
%\begin{equation*}
\begin{align*}
||z-h_{j}^{k,1}||&\le  |e_{k+1}^*(z-h_{j}^{k,1})|+2||Q_{k,k+1}(z-h_{j}^{k,1})||\le \\
&\le a+2||Q_{k,k+1}(h_{j}^{k,1})||+2||Q_{k,k+1}(z)||=\\
&=a+2||bv_{j,k+1}||+2||Q_{k,k+1}(z)||=a+2b+2||Q_{k,k+1}(z)||.
\end{align*}
%\end{equation*}
Since $2\le a+2b$ we get (2.d).

\end{enumerate}

\medskip

\noindent By symmetry, the proof of the above properties (2.a)-(2.c) for $H_{j}^{k,2},\, H_{j}^{k,3}, \,H_{j}^{k,4}$ is similar. 

\bigskip

Finally to get a tiling of the entire space $Z_k$  we consider the additional tiles 
$$\{\widetilde{H}^k_n=4n\widehat{e}_{k+1}+T_k:\, n\in \mathbb Z\setminus \{0\}\}$$
with associated points $\widetilde{h}^k_n:=4n\,\widehat{e}_{k+1}$. Then $Q_{k,k+1}(\widetilde{h}^k_n)=0$,
$B_{Z_k}(\widetilde{h}^k_n,r)\subset \widetilde{H}^k_n$. Also  by  \eqref{inequalityforc4},  $||h-\widetilde{h}^k_n||\le |e_{k+1}^*(h-\widetilde{h}^k_n)|
+2||Q_{k,k+1}(h-\widetilde{h}^k_n)||\le 2+2||Q_{k,k+1}(h)||$ for all $h\in \widetilde{H}^k_n$.

Thus the family $$\{H^k_0,H^{k,p}_j,\widetilde{H}^k_n:\,j\ge 0; \,p=1,2,3,4;\, n\in \mathbb Z\setminus \{0\}\}$$ is a tiling of $Z_k$.
We relabel the tiles of this tiling and their associated points by $\{H^k_j\}$ and $\{h^k_j\}$ respectively, and in such a way that $H^k_0$ remains the same tile. Then, this tiling has conditions (2.a)-(2.d). Notice that $H^k_0$ is a bounded neighborhood of  $0$.
\phantom\qedhere
\end{proof}

\bigskip

\begin{Step} \cite[Theorem 9]{Preiss} \label{tilingwithtubes} There is a tiling $\{C^k_j\}_{j,k}$ of $X$ (where $k>0 \Rightarrow j>0$)  and associated points $\{x^k_j\}_{j,k}\subset
X$ satisfying $$B(x^k_j+V_k,r)\subset C^k_j\subset B(x^k_j+V_k,R)$$
with $x^0_0=0$ and $R:=a+2b+\frac{8}{\delta}$, (for the constants $a,b$, \,$\delta$ and $r$  considered in the previous steps).
\end{Step}
\begin{proof}%\let\qed\relax
Let us consider the tiling $\{H_j^k\}_{j\ge 0}$ of $Z_k$ given in Step  \ref{1scaling} for every $k\ge 0$ and define the sets, where $k>0 \Rightarrow j>0$, as follows  
\begin{equation*}
C^k_j= Q_k^{-1}(H^k_j)\cap \bigcap_{m=k+1}^\infty Q_m^{-1}(H^m_0).
\end{equation*} 
The family $\{C^k_j\}$ satisfies:

\begin{enumerate}
\item[{\em (3.a)}] It is a covering of $X$: for every  $x\in X$  we consider the smallest index $k$ such that $Q_m(x)\in H^m_0$  for all $m>k$.
Notice that $||Q_k(x)||_{Z_k}=\dist(x,V_k) \rightarrow 0$  and thus $x\in B_{Z_k}(0,1)\subset H^k_0$ for $k$ large enough.
If $k=0$ then there is $j\ge 0$ such that $Q_0(x)=x\in H^0_j$ and thus $x\in C^0_j$.
If $k>0$ then  $Q_{k}(x)\notin  H^k_0$ and thus there is $j>0$ such that $Q_k(x)\in H^k_j$. Therefore  $x\in C^k_j$ with $j>0$.

\item[{\em (3.b)}] The sets $C^k_j$ have pairwise disjoint interiors. Indeed, suppose that there is an open ball 
$\overset{\circ}{B}(x,\rho)$
 (with $\rho>0$) such that $\overset{\circ}{B}(x,\rho)\subset C^k_j\cap C^k_{j'}$ for some $j\not=j'$. Then $Q_k(
\overset{\circ}{B}(x,\rho))\subset H^k_{j}\cap
H^k_{j'}$. Since $Q_k$ is continuous, linear and surjective, it is open and thus  $Q_k(\overset{\circ}{B}(x,\rho))$ is a nonempty open set contained 
in  $H^k_{j}\cap
H^k_{j'}$, which is impossible, because $H^k_{j}\cap
H^k_{j'}$ has empty interior. Now, suppose that there is an open  ball $\overset{\circ}{B}(x,\rho)$
 (with $\rho>0$) such that $\overset{\circ}{B}(x,\rho)\subset C^{k}_j\cap C^{k'}_{j'}$ for some $k<k'$. Then, $Q_k(
\overset{\circ}{B}(x,\rho))\subset
H^k_j$ and $Q_m(\overset{\circ}{B}(x,\rho))\subset H^m_0$ for all $m>k$. In particular, $Q_{k'}(\overset{\circ}{B}(x,\rho))\subset H^{k'}_0$. Since $0\le k<k'$, we have $k'>0$ and thus the associated $j'>0$. Also $Q_{k'}(\overset{\circ}{B}(x,\rho))\subset H^{k'}_{j'}$ and thus 
the nonempty open set $Q_{k'}(\overset{\circ}{B}(x,\rho))\subset H^{k'}_0 \cap H^{k'}_{j'}$, which is impossible because 
$H^{k'}_0 \cap H^{k'}_{j'}$ has empty interior.

\item[{\em (3.c)}] Select points $x^k_j \in X $ such that $x^k_j+V_k=Q_k^{-1}h^k_j$. Notice that in particular we can choose
 $x^0_0=0$.
 Let us check that
\begin{equation*}
B(x^k_j+V_k,r)\subset C^k_j\subset B(x^k_j+V_k,R).
\end{equation*}

\noindent Firstly, From condition {\em (2.d)} we have $B_{Z_k}(h^k_j,r)\subset H^k_j$ and thus 
\begin{equation*}
\begin{split}
Q_k^{-1}(B_{Z_k}(h^k_j,r))&=\{x\in X:\, \dist(x-x^k_j,V_k)\le r\}=\\
&=x^k_j+V_k+B(0,r)=
B(x^k_j+V_k,r)\subset 
Q_k^{-1}( H^k_j).
\end{split}
\end{equation*}
Moreover, from condition {\em (2.b)} and for $m>k$ we have that
 $||Q_m(x^k_j)||\le ||Q_{k+1}(x^k_j)||=||Q_{k,k+1}(h_{j}^k)||\le 1-r$. Therefore
\begin{equation*}
Q_m(x^k_j+V_k+B(0,r))=Q_m(x^k_j)+B_{Z_m}(0,r)\subset B_{Z_m}(0,1)\subset H^m_0.
\end{equation*}
Thus $x^k_j+V_k+B(0,r) \subset Q_m^{-1}(H^m_0)$. This yields  $B(x^k_j+V_k,r)\subset C^k_j$.

\medskip

\noindent Secondly, if $x \in C^k_j$ then $Q_{k+1}(x)\in H^{k+1}_0$ and by condition {\em (2.a)} we have
 $||Q_{k+1}(x)||\le\frac{4}{\delta}$. Since $Q_k(x)\in H^k_j$ and applying condition {\em (2.d)} we have that 
$||Q_k(x-x^k_j)|| \le a+2b+2||Q_{k+1}(x)||\le a+2b+2 \frac{4}{\delta}=a+2b+\frac{8}{\delta}$. This yields
\begin{equation*}
 C^k_j\subset B(Q_k^{-1}h^k_j,R)=B(x^k_j+V_k,R)=x^k_j+V_k+B(0,R) \quad \text{ with }R=a+2b+\frac{8}{\delta}.
\end{equation*}
\end{enumerate}
\phantom\qedhere
\end{proof}

\medskip

 \begin{Step} \label{refinement} There is a tiling $\{C_{i,j}^k\}_{i,j,k}$ of $X$ (where ``$k>0 \Rightarrow j>0$'' and ``$k=0\Rightarrow
i=0$'') and associated  points $\{d^{k}_i+x^k_j\}_{i,j,k}\subset X$ satisfying
\begin{equation*}
B(d^{k}_i+x^k_j,r)\subset C^k_{i,j}\subset B(d^{k}_i+x^k_j, R'),
\end{equation*}
with $d^{0}_0=x^0_0=0$ and $R'=a+2b+\frac{8}{\delta}+2r$, (for the constants $a,b$, \,$\delta$ and $r$  considered in the previous steps).
\end{Step}
\begin{proof} 
In the  last step of the proof of Theorem \ref{separablecase} we are going to define a refinement of the  tiling  
of $X$ obtained in Step \ref{tilingwithtubes}  by tiling each  tile.
 Note that each tile is a sort of  ``cylinder'' around  a  finite 
dimensional subspace $V_k$. In order to do that, first consider a family $\{d^{k}_i\}_i\subset V_{k}$
being maximal in  $V_{k}$ with respect to the property $||d^{k}_i-d^{k}_{i'}||\ge 2r$ for $i\not=i'$. Notice that for $k=0$ we have 
$V_0=\{0\}$ and only one point $d^0_i$ which is $d^0_0=0$. For $k>0$, $i$ runs over $\mathbb N\cup\{0\}$. For every $k\ge 0$ fixed, consider the ``modified
Voronoi cells'' $\{D_i^{k}\}_i$
in $X$ associated with the set $\{d^{k}_i\}_i$ and defined by

\begin{equation*}
\begin{split}
V_i^{k}&=\{x\in X:\,||x-d^{k}_i||\le ||x-d^{k}_{i'}||\text{ for all } i'\},\\
D_i^{k}&=\overline{(V_i^k\setminus \bigcup_{i'<i}V_{i'}^k)}\subset V_{i}^{k}.
\end{split}
\end{equation*}

\bigskip

\begin{enumerate}

\item[{\em (4.a)}] 
 For every $k$ fixed, the family $\{D_i^{k}\}_i$ is a tiling of $X$.  The proof is similar to the one  given in \cite[Lemma 2.2]{DevilleGarciaBravo}) and \cite[Theorem 3.1]{Klee2}. For $k=0$, there is just one tile $D^0_0=X$. For $k>0$, since the family $\{d^k_i\}_i\subset V_k$ is $2r$-separated and $V_k$ is finite dimensional, the limit $\lim_{i\to \infty}||d^k_i||=\infty$ and thus for every $x\in X$  
the $\inf_i||x-d^k_i||$ is attained. Then $x\in D^k_{i}$ being   $i$ the first $i$ where the infimun is attained.
Thus, the family $\{D^k_i\}_i$ covers $X$.
Also, the sets of the family $\{D_i^{k}\}_i$ have pairwise disjoint interiors. Indeed, if $i'<i$ and there is an open ball 
$\overset{\circ}{B}(x,\rho)\subset D_i^{k}\cap D_{i'}^{k}$ (with $\rho>0$). 
Then $\overset{\circ}{B}(x,\rho)\subset D_{i'}^k\subset V_{i'}^k$
and  $\overset{\circ}{B}(x,\rho)\subset D_{i}^k\subset \overline{V_{i}^k\setminus V_{i'}^k}$. Thus, 
$V^k_i\setminus \overset{\circ}{B}(x,\rho)\supset V_i^k\setminus V_{i'}^k$. Since the set 
$V^k_i\setminus \overset{\circ}{B}(x,\rho)$ is closed we have 
$V^k_i\setminus \overset{\circ}{B}(x,\rho)\supset \overline{ V_i^k\setminus V_{i'}^k}$. This is a contradiction with the fact that $\overset{\circ}{B}(x,\rho)\subset \overline{V_{i}^k\setminus V_{i'}^k}$.

\medskip

\item[{\em (4.b)}] In addition, let us see that the sets of the family $\{D_i^{k}\}_i$ are starlike. In fact, we will check that $D_i^{k}$ is starlike   with respect to $d^{k}_i$. Since the closure of a starlike set is starlike, it is enough to check that
 $V_i^k\setminus \bigcup_{i'<i}V_{i'}^k$ is starlike. If $x\in V_i^k\setminus \bigcup_{i'<i}V_{i'}^k$, then
\begin{equation*}
||x-d^k_i||\le ||x-d^k_{i'}|| \quad \text{ for all } i'  \quad \text{ and } \quad ||x-d^k_i||< ||x-d^k_{i'}|| \quad \text{ for }\quad
 i'<i. 
\end{equation*}
For all $t\in[0,1]$ we have that $x_t=tx+(1-t)d^k_i$ satisfies
\begin{equation*}
||x_t-d^k_i||=||x-d^k_i||-||x-x_t|| \quad \text{ and } \quad ||x-d^k_{i'}|| \le ||x-x_t||+||x_t-d^k_{i'}|| \quad \text{ for all } \quad i'.
\end{equation*}
Thus,
%\begin{equation*}
\begin{align*}
||x_t-d^k_i||&=||x-d^k_i||-||x-x_t|| \le ||x-d^k_i|| +||x_t-d^k_{i'}||-||x-d^k_{i'}|| \le ||x_t-d^k_{i'}|| \text{ for all } i',\\
||x_t-d^k_i||&=||x-d^k_i||-||x-x_t|| \le ||x-d^k_i|| +||x_t-d^k_{i'}||-||x-d^k_{i'}|| < ||x_t-d^k_{i'}|| \text{ for } i'<i.
\end{align*}
%\end{equation*}
Thus  $x_t\in V_i^k\setminus \bigcup_{i'<i}V_{i'}^k$.

\medskip

Also,  it is worth noting that in the case that the norm is strictly convex we could consider the sets $V_i^k$
in place of the sets $D_i^k$ because the sets $V_i^k$ have pairwise disjoint interior (\cite[Theorem 3.1]{Klee2}). 
Also  if the norm is
locally uniformly rotund then the sets 
 $D_i^{k}$ coincide with the sets $V_i^{k}$ because for any fixed $k$ the sets of the family $\{V_i^{k}\}_i$ have pairwise 
 disjoint interiors and each $V_i^{k}$
is the closure  of its interior (\cite[Theorem 3.1]{Klee2}).

\medskip

\item[{\em (4.c)}] Moreover, $B(d_i^{k},r)\subset D_i^{k}$. Indeed, if $x\in X$ and $||x-d_i^{k}||< r$ then for $i'\not=i$
we have 
$$||x-d^k_{i'}||=||(d^k_i-d^k_{i'})+(x-d^k_i)||\ge ||d^k_i-d^k_{i'}||-||x-d^k_i||> 2r-r=r> ||x-d_i^{k}||$$ 
and 
$$x\in V^k_i\setminus \bigcup_{i'\not= i}V^k_{i'}\subset V^k_i\setminus \bigcup_{i'<i}V^k_{i'}. $$
Thus
the closed  ball $$B(d^k_i,r)\subset \overline{V^k_i\setminus \bigcup_{i'<i}V^k_{i'}}=D^k_i.$$

\medskip

\item[{\em (4.d)}] Take the intersection  $C_{i,j}^k=(x^k_j+D_i^{k})\cap C^k_j$  and the family $\{C_{i,j}^k\}_{i,j,k}$
(where ``$k>0\Rightarrow j>0$'' and ``$k=0 \Rightarrow i=0$''). Let us see some of the properties of this family:

\medskip

\noindent $\star$ The family $\{C_{i,j}^k\}_{i,j,k}$ covers $X$: if $x\in X$, then there is  $C^k_j$ such that $x\in C^k_j$ (because
$\{C^k_j\}_{j,k}$ is a covering of $X$. Now, for $k,j$ fixed, $\{(x^k_j+D_i^{k})\cap C^k_j\}_i$ is a covering of $C^k_j$ and thus
there is $i$ such that $x\in (x^k_j+D_i^{k})\cap C^k_j$.

\medskip

 \noindent $\star$ Each set $C_{i,j}^{k}$ is starlike with respect to $x^k_j+d_i^{k}$:
  Suppose that $x\in C_{i,j}^{k}$.
Since  $C_{j}^{k}$ is convex,  $x\in C_{j}^{k}$ and $x^k_j+d_i^{k}\in x^k_j+V_k\subset C_{j}^{k}$, we have that  
$t (x^k_j+d_i^{k})+(1-t)x\in C_j^k$ for every $0\le t\le 1$.

\noindent  Also, the translated set $x^k_j+D_i^{k}$ is starlike with respect to $x^k_j+d_i^{k}$ and thus 
$t (x^k_j+d_i^{k})+(1-t)x\in x^k_j+D_i^{k}$ for all $0\le t\le 1$.
Therefore  $t (x^k_j+d_i^{k})+(1-t)x\in C_{i,j}^k$ for all $0\le t\le 1$.

\bigskip

\noindent $\star$  ${B(x^k_j+d_i^{k},r)\subset C_{i,j}^k}$. Indeed, from Step \ref{tilingwithtubes} we get
$$B(x^k_j+d_i^{k},r)=x^k_j+d_i^{k}+B(0,r)\subset x^k_j+V_k+B(0,r)\subset C^k_j.$$

\noindent  Also, from (4.c) we get $$B(x^k_j+d_i^{k},r)=x^k_j+B(d_i^{k},r)\subset x^k_j+D_i^k.$$

\bigskip

\noindent $\star$ ${C^k_{i,j}\subset B(x^k_j+d_i^{k},R+2r)}$. Indeed, suppose  there is $x\in C_{i,j}^k$ satisfying that $||x-(x^k_j+d_i^{k})||> R+2r$. 
Since $x\in C^k_j\subset x^k_j+V_k+B(0,R)$, there is a vector  $v\in V_{k}$  such that 
$||x-(x^k_j+v)||\le R$. For this $v\in V_k$ there is $d_{i'}^k\in V_k$ such that $||v-d_{i'}^k||\le 2r$. Now,

\begin{equation*}
\begin{split}
||(x-x^k_j)-d_{i'}^{k}||&=||x-x^k_j-d_{i'}^{k}||= ||x-x^k_j-v+v-d_{i'}^{k}||\le ||x-x^k_j-v||+||v-d_{i'}^{k}||\le R+2r<\\
&<||x-(x^k_j+d_{i}^{k})||=||(x-x^k_j)-d_{i}^{k}||.
\end{split}
\end{equation*}

\medskip

\noindent Therefore $x-x^k_j\notin V_i^k$ and thus $x-x^k_j\notin D_i^k$, i.e. $x\notin x^k_j+ D_i^k$ and
 $x\notin (x^k_j+ D_i^k)\cap C^k_j=C^k_{i,j}$, which is not true.

\bigskip

\noindent \noindent $\star$ The sets of the family $\{C_{i,j}^{k}\}_{i,j,k}$ have pairwise disjoint interiors. Indeed, this assertion follows from the fact that every $k\ge 0$ fixed,
the  sets $\{{D}_i^{k}\}_i$  have pairwise disjoint interiors and the sets $\{{C}_j^{k}\}_{j,k}$ have pairwise disjoint interiors.
Suppose there is a ball $B(x,s)\subset C_{i,j}^{k}\cap C_{i',j'}^{k'}$ with $x\in X$ and $s>0$.

\medskip

\noindent  (a) If $(j,k)\not=(j',k')$, 
then $B(x,s)\subset
C^k_{j}\cap C^{k'}_{j'}$ and thus  the interior of $C^k_{j}\cap C^{k'}_{j'}\not=\emptyset$, which is not true because of (3.b).

\medskip

\noindent (b) If $(j,k)=(j',k')$ and $i\not=i'$ then $B(x,s)\subset (x^k_j+D_i^k)\cap (x^k_j+D_{i'}^k)=x^k_j+D_i^k\cap D_{i'}^k$ and thus the interior of
$D_i^k\cap D_{i'}^k\not=\emptyset$, which is not true because of  (4.a).

\end{enumerate}

\noindent Notice that the constant of normality $K$  can be chosen any number 
$\displaystyle{K>\frac{R+2r}{r}=2+\frac{1}{r}(a+2b+\frac{8}{\delta})}.$
This finishes the proof of Theorem \ref{separablecase}.  
\end{proof} 

\bigskip

\begin{Rem}\label{k=0}
Note that in the proof of Theorem \ref{separablecase}, for $k=0$ the vector space $V_0=\{0\}$ and the sets $C^0_j$ verify that $B(x^0_j,r)\subset C_{j}^0\subset B(x^0_j,R)$.
Also, there is just one $d^0_i$ which is $d^0_0=0$.
Therefore the tiles $C_{j}^0$ of Step \ref{tilingwithtubes} remain the same after the refinement given in Step \ref{refinement} and thus, the tiles
$C_{0,j}^0=C^0_j$ are convex and satisfy  $B(x^0_j,r)\subset C_{0,j}^0=C_{j}^0 \subset B(x^0_j,R)$. In particular, since $x_0^0=0$ we get
that $B(0,r)\subset C^0_{0,0}=C^0_{0}=B(0,R)$ and thus the ``first tile'' $C^0_{0,0}$ is a bounded convex neighborhood of $0$.
\end{Rem}

\medskip

\begin{Rem}\label{constantbound}
Let us give some estimations of the constant  of normality $K$. For example, for $a=1.3$, $b=0.9$, $r=0.1$ and $\delta= \frac{5}{9}$
we get $\frac{R+2r}{r}=177$. Thus, we get a constant of normality $K>177$. For other choices of $U_0$ in Lemma \ref{R2} it is possible to obtain better constants of normality. For example, let us redefine $U_0$ as 
$U_0=\{(x,y)\in \mathbb R^2:\, |x|+2|y|\le 3,\,\, |x|\le 2\}$ 
and  modify conveniently the sets $U_i$ for $i=1,2,3,4$. That allows to have $H^k_0\subset B(0,\frac{3}{\delta})$ for all $k$,
which yields to a constant $R=a+2b+\frac{6}{\delta}$. 
 Now, for example for $a=1.8$, $b=0.8$, $r=\frac{0.4}{3}$ and $\delta=0.5$, we get 
$\frac{R+2r}{r}=117.5$ and thus a constant of normality $K>117.5$.
\end{Rem}

\medskip

In the next Corollary, instead of  using ``modified Voronoi cells'' to slice every ``cylinder'' $C^k_j$ of Theorem \ref{separablecase}
we could fix a natural number $N$  and use the  linear projections $P_k:X\rightarrow V_k$, \, $P_k(x)=\sum_{m=1}^ke_m^*(x)e_m$ for $x\in X$ and $0<k\le N$  to slice the tiles $C^k_j$ for $0<k\le N$ and all $j$. For $k=0$ and $k>N$ we use the 
``modified Voronoi cells'' as before. Since the constants $R_N:=\sup_{0<k\le N}||P_k||<\infty$, the refinement of the tiling
 $\{C^k_j\}$ would produce a normal tiling $\{\widetilde{C}^k_{i,j}\}$ defined in the following propositon.

\begin{Pro} \label{separableconvex} Let $(X,||\cdot||)$ be a separable infinite dimensional Banach space. Following the notation in the proof of Theorem \ref{separablecase}, consider for any fixed natural number  $N$  the linear projections \linebreak
$P_k:X\rightarrow V_k$, $P_k(x)=\sum_{m=1}^ke_m^*(x)e_m$ for  $0<k\le N$  and  all $x\in X$. 
Then, there is a tiling $\{\widetilde{C}^k_{i,j}\}_{i,j,k}$ of $X$ (where ``$k>0 \Rightarrow j>0$''
and ``$k=0\Rightarrow i=0$'') so that:

\begin{enumerate}
\item[{ (a)}] The tiles $\widetilde{C}^k_{i,j}$ are convex for $k\le N$ and all $i,j$. In particular for $0<k\le N$,
$\widetilde{C}^k_{i,j}=P_{k}^{-1}(\widetilde{D}_i^k)\cap C^k_j$
where $\{\widetilde{D}_i^k\}_i$ is a normal and convex tiling of $V_k$. Also, $\widetilde{C}^0_{0,j}=C^0_j$ for all $j$ and   the tile $C^0_{0,0}=C^0_0$ is a bounded and convex neighborhood of $0$.

\item[{ (b)}] The tiles $\widetilde{C}^k_{i,j}=C^k_{i,j}=(x^k_j+D^k_i)\cap C^k_j$ are starlike for $k> N$ and all $i,j$,
where the points $\{x^k_j\}_{j,k}$ are defined in Step \ref{tilingwithtubes} (of Theorem \ref{separablecase}) and the normal and starlike tiling $\{D_i^k\}_i$ of $V_k$ is defined in Step \ref{refinement} (of Theorem \ref{separablecase}).

\item[{ (c)}] The tiling  $\{\widetilde{C}^k_{i,j}\}_{i,j,k}$ of $X$ is normal.
\end{enumerate}
\end{Pro}
\begin{proof}
 (a) Let us consider the tiling $\{C^k_j\}_{j,k}$ of $X$ constructed in Step \ref{tilingwithtubes} in the proof of Theorem \ref{separablecase} and follow  the notation of that proof. Define $\widetilde{C}^k_{i,j}=P_{k}^{-1}(\widetilde{D}_i^k)\cap C^k_j$  
  for $0<k\le N$ and all $i,j$ 
where $\{\widetilde{D}_i^k\}_i$ is a normal and convex tiling of $V_k$ such that $B_{V_k}(\widetilde{d}^k_i,r)\subset
\widetilde{D}_i^k \subset B_{V_k}(\widetilde{d}^k_i,r_k)$ for some $r_k\ge r$ and suitable points $\widetilde{d}_i^k\in V_k$. Since $\{e_m\}_m$ is a fundamental biorthogonal system of $X$, we have that $X=[\{e_m\}_{m=1}^k]\oplus [\{e_m\}_{m=k+1}^\infty]$, where
$[H]$ denotes the closed linear span of the set $H$. Thus,  can assume that the points $x^k_j$  in the proof of the Theorem 
\ref{separablecase} have been choosen so that $x^k_j\in [\{e_i:\,i>k\}]$ for $0<k\le N$. Now if  $x\in  \widetilde{C}^k_{i,j}$, let us check that
\begin{equation*}
B(x^k_j+\widetilde{d}^k_i,\frac{r}{||P_k||})\subset \widetilde{C}^k_{i,j}\subset B(x^k_j+\widetilde{d}^k_i,\,R(1+||P_k||)+r_k).
\end{equation*}
Recall that $R=a+2b+\frac{8}{\delta}$, where $a,b,\delta$ are the positive constants chosen in Lemma \ref{R2}.
Firstly, take $x\in  \widetilde{C}^k_{i,j}\subset C^k_j\subset x^k_j+V_k+B(0,R)$ and  $v\in V_k $ such that $||x-x^k_j-v||=\dist(x-x^k_j, V_k)\le R$. Then,
%\begin{equation*} 
\begin{align*}
||x-(x^k_j+\widetilde{d}^k_i)||&=||(x-x^k_j-v)+ (v-\widetilde{d}^k_i)|| \le  R+||(v-P_k(x))+(P_k(x)-\widetilde{d}^k_i)||\le \\
& R+ ||P_k(v-x)||+r_k = R+ ||P_k(v-x+x^k_j)||+r_k\le \\ & \le   R+||P_k||R+2r_k\le R(1+||P_k||)+r_k.
\end{align*}
%\end{equation*}

\medskip

Secondly, if $x\in B(x^k_j+\widetilde{d}^k_i,\frac{r}{||P_k||})$, then $P_k(x)\in B_{V_k}(\widetilde{d}^k_i,r)$ and 
$Q_k(x)\in B_{Z_k}(\widehat{x}^k_j,\frac{r}{||P_k||})\subset B_{Z_k}(\widehat{x}^k_j,r)$. Thus, 
$P_k(x)\in \widetilde{D}^k_i$ and
$x\in B(x^k_j+V_k,r)\subset C^k_j$. This yields $x\in \widetilde{C}_{i,j}^k$.
Therefore if $t_N:=\sup \{r_k:\,0<k\le N\}$ and $R_N:=\sup_{0<k\le N}||P_k||<\infty$ we have
\begin{equation}\label{normal1} 
B(x^k_j+\widetilde{d}^k_i,\frac{r}{R_N})\subset \widetilde{C}^k_{i,j}\subset B(x^k_j+\widetilde{d}^k_i,\,R(1+R_N)+t_N)
\quad  \text{ for } \quad  0< k\le N.
\end{equation}

\medskip

 For $k=0$ and $k>N$, the sets $\widetilde{C}^k_{i,j}$ are the sets ${C}^k_{i,j}$ defined in the proof of Theorem 
\ref{separablecase}. Also, the points $x^k_j$ and ${d}^k_i$ are the points defined in the proof of Theorem 
\ref{separablecase}. Recall that they satisfy
\begin{equation}\label{normal2} 
B(x^k_j+d^k_i,r)\subset \widetilde{C}^k_{i,j}\subset B(x^k_j+{d}^k_i,\,R+2r)
\quad  \text{ for } \quad  k=0 \text{ and } k> N.
\end{equation}

\medskip 

The fact that $\{\widetilde{C}^k_{i,j}\}_{i,j,k}$ is a tiling of $X$ can be proved as in Theorem \ref{separablecase}.
Assertion (b) is a direct consequence of the proof of Theorem \ref{separablecase}.
Because of the fact that $\displaystyle{\frac{r}{R_N}\le r < R+2r<R+2 < R+R\le R(1+R_N)< R(1+R_N) +t_N }$ 
and the inclusions in \eqref{normal1} and \eqref{normal2}, it follows  that the tiling $\{\widetilde{C}^k_{i,j}\}_{i,j,k}$
is $K$-normal for any constant $\displaystyle{K>\frac{R_N}{r}(R(1+R_N) +t_N)}$  and  assertion (c) is proved.
\end{proof} 

\medskip

\begin{Rem} A result of M.I. Kadec and M.G. Snobar \cite[pg. 243 and pg. 320]{BST} establishes that for any $n$-dimensional space $V$ and any Banach space $X$ containing $V$ the constant $$\lambda(V,X)=\inf\{ ||P|| : \, P:X\rightarrow V \text{ is a linear projection 
 from } X \text{ to } V\} \le \sqrt{n}.$$
Thus,  replacing the projections $\{P_k:X\rightarrow V_k\}_{k=1}^N$
 with suitable projections $\{\widetilde{P}_k:X\rightarrow V_k\}_{k=1}^N$ in Proposition \ref{separableconvex},
we can get  
$R_N$  as close  as needed to $\sqrt{N}$.
Also recall that in the introduction it was mentioned that any $n$-dimensional space $V$ has a $2\sqrt{n}$-normal convex  tiling and thus  the constant $t_N$ in the proof of Proposition \ref{separableconvex} can be taken to satisfy $t_N\le 2\sqrt{N}$.
  In this case (by modifying the  points $x^k_j$ for $0<k\le N$ to be in $\ker \widetilde{P}_k$) a coarse estimation for the constant of normality 
 $K>\frac{R_N}{r}(R(1+R_N) +t_N)$ given in the proof of Proposition \ref{separableconvex} will be  
$K>\frac{1}{r}[R(N+\sqrt{N}) +2N]$ with $R=a+2b+\frac{8}{\delta}$.
\end{Rem}

Recall  that in any finite dimensional space $(X,||\cdot||)$ there is a $2$-normal and starlike tiling  \cite{DevilleGarciaBravo}. Recall also that it is an open problem whether there is $S>0$ such that every finite dimensional Banach space $(X,||\cdot||)$
has $S$-normal and convex tilings. We finish this section with a review of the  steps \ref{1scaling} to \ref{refinement} in the proof of Theorem \ref{separablecase} for the finite dimensional setting. We adapt it to  get $K$-normal tilings such that some  tiles are convex and the rest of them are starlike being  $K=2+\frac{1}{r}(a+2b+\frac{8}{\delta})$  independently of the dimension of the space.

 \begin{Pro} \label{Ndim} For every space $(X,||\cdot||)$ of dimension $M$  there is a $K$-normal starlike  tiling $\{C^{k}_{i,j}\}_{i,j,k}$ (indexed with  $0\le k\le M-1$, where ``$k>0 \Rightarrow j>0$'' and ``$k=0\Rightarrow i=0$'') such that
\begin{enumerate}
\item[(a)] the tiles $\{C^{0}_{0,j}\}_j$ are convex,\, 
\item[(b)]  the tile $C^{0}_{0,0}$ is a convex neighborhood of $0$,\, 
\item[(c)] the constant $K$ can be taken as $K=2+\frac{1}{r}(a+2b+\frac{8}{\delta})$ independently of $M$ (the constants $a,b, 
r,\delta$ are given by Lemma \ref{R2} in the proof of Theorem \ref{separablecase}).
\end{enumerate}
\end{Pro}
\begin{proof} Let us review the 4 steps of the proof of Theorem \ref{separablecase} for the $M$-dimensional space $X$.
In the following we can assume  $M\ge 2$.
Firstly, we consider an Auerbach basis, i.e. a normalized basis $\{e_i\}_{i=1}^M$ of $X$ with  
associated normalized  biorthogonal functionals $\{e_i^*\}_{i=1}^M$.
We consider the subspaces $V_0=\{0\}$ and $V_k=[\{e_1,\dots,e_k\}]$ for $0<k\le M$,
the quotient spaces $Z_k=X/V_k$ and the  
 natural quotients $Q_k:X\rightarrow Z_k$. Also, consider $Q_{k,k+1}:Z_k\rightarrow Z_{k+1}$  for $k+1\le M$.

\medskip

\begin{enumerate}

\item[(1)] Step \ref{semibetasystem} in the proof of Theorem \ref{separablecase} provides a finite and normalized
family $\{v_{j,k}^*,v_{j,k}\}_{j=0}^{m_k}\subset Z_k^*\times Z_k$ 
 for $0\le k\le M-1$ and certain integer numbers $m_k\ge 0$ satisfying conditions (1.a)-(1.c).

\medskip

\item[(2)] In Step \ref{1scaling} of the proof of Theorem \ref{separablecase}, the surjective functions $\pi_{j,k}:Z_k\rightarrow \mathbb R^2$, 
the sets $T_k,\,H^k_0$, $H^{k,p}_j$, $\widetilde{H}^k_n$  and the associated points $h^{k,p}_j$, $\widetilde{h}^{k}_n$
 are defined in the same way 
for \, $0\le k \le M-2$, \, $p=1,2,3,4$,   \, $n\in \mathbb Z \setminus \{0\}$ \, and \, $0\le j \le m_k$. 
The only difference is that
the sets $H^k_0$ are defined by the finite intersection
$$H_0^k=\bigcap_{0\le j \le m_k}\pi_{j,k}^{-1}(U_0) \quad \text{ for } \quad 0\le k\le M-2.$$
By relabeling $\{H^k_0$, $H^{k,p}_j$, $\widetilde{H}^k_n\}$ and their associated points we get the convex tiling $\{H^k_j\}_{j=0}^\infty$ of $Z_k$ with associated points $\{h^k_j\}_{j=0}^\infty$ (where $h^k_0=0$) satisfying  properties (2.a) to (2.d).

Now, for $k=M-1$, we define 
%\begin{equation*}
\begin{align*}
\pi_{0,M-1}&:Z_{M-1}\rightarrow \mathbb R, \quad  \pi_{0,M-1}(z)=e_{M}^*(z) \quad  \text{ for } z\in Z_{M-1},\\
T_{M-1}&=\{z\in Z_{M-1}:\, |e^*_{M}(z)|\le 2\}=B_{Z_{M-1}}(0,2), \\
\widetilde{H}^{M-1}_n&=4n\widehat{e}_{M}+T_{M-1} \quad \text{ for } n\in \mathbb Z,\\
\end{align*}
%\end{equation*} 
where $\widetilde{H}^{M-1}_n$ has the associated point $4n\widehat{e}_{M}$ for all $n\in \mathbb Z$. Again, by relabeling $\{\widetilde{H}^{M-1}_n\}$ and the associated points, 
we get the convex
 tiling $\{H^{M-1}_j\}_{j=0}^\infty$ of $Z_{M-1}$ with associated points $\{h^{M-1}_j\}_{j=0}^\infty$ (where $h^{M-1}_0=0$) satisfying the properties
\begin{enumerate}
\item[{\em (2.a')}] $B_{Z_{M-1}}(0,2)= H^{M-1}_0$,
\item[{\em (2.b')}] $Q_{M-1,M}(h^{M-1}_j)=0=Q_{M-1,M}(h)$ for all $h\in H_j^{M-1}$,
\item[{\em (2.c')}] $B_{Z_{M-1}}(h^{M-1}_j,2)= H^{M-1}_j$ and thus
\item[{\em (2.d')}] $||h-h^{M-1}_j||\le 2$ for every $h\in H_j^{M-1}$.
\end{enumerate}
This is just a tiling in the $1$-dimensional space $Z_{M-1}$ by closed intervals of length $4$.

\medskip

\item[(3)] Step \ref{tilingwithtubes} in the proof of Theorem \ref{separablecase} provides a tiling $\{C^k_j\}_{j,k}$ of $X$ (where  $0\le k\le M-1$ \, and \, $j\ge 0$; \, also \, ``$k>0 \Rightarrow j>0$'') and associated points $\{x^k_j\}_{j,k}\subset X$ such that
\begin{equation*}
B(x^k_j+V_k,r)\subset C^k_j\subset B(x^k_j+V_k,R) \quad \text{ for } 0\le k\le M-1,\\
\end{equation*}
with same $x^0_0=0$ and same constants $r>0$ and $R=a+2b+\frac{8}{\delta}$. In this case, the only difference is the definition of the sets $C^k_j$ as the finite intersection
\begin{equation*}
C^k_j=Q_k^{-1}(H^k_j)\cap \bigcap_{m=k+1}^{M}Q_m^{-1}(H^m_0), \quad \text{ for } 0\le k\le M-1,
\end{equation*}
where we define $H^M_0=\{0\}$. Clearly, for $k=M-1$ and $j>0$ the sets $C^{M-1}_j= Q_{M-1}^{-1}(H^{M-1}_j)=
\{x\in X:\, |e_{M}^*(x)-4n_j|\le 2\}$ for certain $n_j\in \mathbb Z\setminus \{0\}$. So for $k=M-1$ and $j>0$, 
\begin{equation*}
C^{M-1}_j=B(x^{M-1}_j+V_{M-1},2)=x^{M-1}_j+V_{M-1}+B(0,2).
\end{equation*}

\medskip
\item[(4)] Step \ref{refinement} in the proof of Theorem \ref{separablecase} provides a starlike tiling 
$\{C_{i,j}^k\}_{i,j,k}$ of $X$ (where  $0\le k\le M-1$ \, and \, $j\ge 0$; \, also \, ``$k>0 \Rightarrow j>0$'' and ``$k=0 \Rightarrow i=0$'') and associated points $\{d^k_i+x^k_j\}\subset X$ satisfying
\begin{equation} \label{normalNdim}
B(d^k_i+x^k_j,r)\subset C_{i,j}^k\subset B(d^k_i+x^k_j,R'),
\end{equation}
with  same points $d^0_0=x^0_0=0$ and same constants $r$ and $R'=a+2b+\frac{8}{\delta}+2r$. 
Indeed, in order to get the above tiling, the points $\{d^k_i\}_i\subset V_k$, the sets $\{V^k_i\}_i$ and $\{D^k_i\}_i$ are defined in the same way for $0\le k\le M-1$. Moreover, the tiles $C_{i,j}^k=(x^k_j+D^k_i)\cap C^k_j$ are defined in the same way for 
$0\le k\le M-1$. 

Finally, this tiling $\{C_{i,j}^k\}_{i,j,k}$ verifies that the tiles $C_{0,j}^0=(x^0_j+D^0_0)\cap C^0_j=C^0_j$ are convex for all $j$ and the tile $C^0_{0,0}$ is a convex neighborhood of $0$ because $d^0_0+x^0_0=0$. So we get (a)  and (b) in the statement of Proposition \ref{Ndim}.

From the inclusions in \eqref{normalNdim}, we can take the  constant of normality as $K=2+\frac{1}{r}(a+2b+\frac{8}{\delta})$ and get (c) in the statement of Proposition \ref{Ndim}.
\end{enumerate}
\end{proof}

\begin{Rem} \label{finitedimensionalconvex} An analogous result to Proposition \ref{separableconvex} can be established in the case of any finite dimensional space $(X,||\cdot||)$
if we fix a natural number $N$ and consider any space of dimension $M > N$. By slicing the ``cylinders'' $C^k_j$ in the proof  of Proposition \ref{Ndim} for $0<k\le N$ with similar  projections $P_k$ to those given in the proof of Proposition \ref{separableconvex} and for $k>0$ and $N<k<M$ by using ``modified Voronoi cells'' as in the proof of Theorem \ref{separablecase}, we get a  normal tiling $\{\widetilde{C}^k_{i,j}\}_{i,j,k}$ (labelled in the same way as in the proof  of Proposition \ref{Ndim}) with normality constant not depending on the dimension $M> N$ (although it depends on $N$), where  the tiles $C^k_{i,j}$ are convex for $0\le k\le N$ and the rest of them are  starlike.
\end{Rem}

\begin{Rem} It is worth mentioning a result of G. Pisier \cite{Pisier}, asserting that there is a separable infinite dimensional Banach space $(X,||\cdot||)$ and a number $s>0$ such that  every finite rank projection $P$ from $X$ into $X$ satisfies $||P||\ge s\sqrt{\rank P}$.
So in this particular situation it is not clear how to adapt the construction of Preiss by using projections over finite dimensional subspaces of $X$.
\end{Rem}

\bigskip

\noindent {\bf Acknowledgments.} M. Jimenez-Sevilla wishes to thank the Institut de Math\'ematiques de la Universit\'e de Bordeaux, where  this work was carried out.

\end{document}